\documentclass[12pt,oneside,reqno]{amsart}       

\usepackage{etex}
\reserveinserts{28}  

\usepackage{amssymb,amsthm,amsmath,mathrsfs,comment}
\usepackage{amstext,amsxtra}  
\usepackage{fullpage,colonequals}
\usepackage{bm}       
\usepackage{mathtools} 
\usepackage[all]{xy}
\usepackage{enumitem}

\usepackage{pictex}
\usepackage{multirow}

\usepackage[pdfpagelabels]{hyperref}
\hypersetup{pdftitle={Signature Ranks of Units in Cyclotomic Extensions of Abelian Number Fields},pdfauthor={David S.\ Dummit}} 
\hypersetup{colorlinks=true,linkcolor=blue,anchorcolor=blue,citecolor=blue}

\theoremstyle{plain}

\newtheorem{theorem}{Theorem}

\newtheorem{proposition}{Proposition}

 \newtheorem*{nonumlemma}{Lemma}                              

\newtheorem{corollary}{Corollary}

\theoremstyle{remark}
\newtheorem{remark}{Remark}

\def\ZZ{\mathbb Z}
\def\QQ{\mathbb Q}
\def\RR{\mathbb R}
\def\CC{\mathbb C}
\def\FF{\mathbb F}

\def\epsilon{\varepsilon}
\def\phi{\varphi}

\DeclareMathOperator{\iso}{\simeq}
\DeclareMathOperator{\Gal}{Gal}
\DeclareMathOperator{\Sel}{Sel}
\DeclareMathOperator{\sgn}{sgn}
\DeclareMathOperator{\sign}{sign}

\begin{document}

\title[Signature Ranks of Units in Cyclotomic Extensions of Abelian Number Fields]
{Signature Ranks of Units in Cyclotomic Extensions of Abelian Number Fields}

\author{David S.\ Dummit}
\address{David S.\ Dummit, Department of Mathematics, University of Vermont, Lord House, 16 Colchester Ave., Burlington, VT 05405, USA}
\email{dummit@math.uvm.edu}

\author{Evan P.\ Dummit}
\address{Evan P.\ Dummit, School of Mathematical and Statistical Sciences, Arizona State University, P.O.~Box 871804, Tempe, AZ 85287, USA}
\email{Evan.Dummit@asu.edu}

\author{Hershy Kisilevsky}
\address{Hershy Kisilevsky, Department of Mathematics and Statistics and CICMA, Concordia University, 1455 de Maisonneuve  Blvd. West, Montr\'eal, Quebec, H3G 1M8, CANADA}
\email{hershy.kisilevsky@concordia.ca}


\begin{abstract}
We prove the rank of the group of signatures of the circular units 
(hence also the full group of units) of 
$\QQ( \zeta_m)^+$ tends to infinity with $m$. 
We also show the signature rank of the units differs from its maximum possible value
by a bounded amount for all the real subfields of the composite of an abelian field with
finitely many odd prime-power cyclotomic towers.  In particular, for any prime 
$p$ the signature rank of the units of 
$\QQ( \zeta_{p^n})^+$ differs from $\phi(p^n)/2$ by 
an amount that is bounded independent of $n$.  Finally, we show conditionally that
for general cyclotomic fields
the unit signature rank can differ from its
maximum possible value by an arbitrarily large amount.
\end{abstract}

\subjclass[2010]{11R18 (primary), 11R27 (secondary)}

\maketitle

\section{Introduction}

For any positive integer $m$ with $m$ odd or $m$ divisible by 4, let $\zeta_m$ be a primitive
$m$-th root of unity, $\QQ(\zeta_m)$ the corresponding cyclotomic field and
$\QQ(\zeta_m)^+ = \QQ(\zeta_m + \zeta_m^{-1})$ its maximal totally real subfield.
The units in $\QQ(\zeta_m)^+$ 
together with $\zeta_m$ are a subgroup of finite index in the group of units of $\QQ(\zeta_m)$ (this
index is 1 when $m$ is a prime power and 2 otherwise, \cite[Corollary 4.13]{W}).

Under the various $\phi(m)/2$ real embeddings of $\QQ(\zeta_m)^+$, each unit $\epsilon$ of $\QQ(\zeta_m)^+$ has a sign, and the
collection of these signs is called the {\it signature} of $\epsilon$.  The collection of all such unit
signatures is an elementary abelian 2-group, whose rank is called the {\it unit signature rank} of $\QQ(\zeta_m)^+$
(or, by abuse, of $\QQ(\zeta_m)$).  The signature rank measures how many different possible signs arise from
the units and determines the difference between the class number 
and the strict class number.

The purpose of this paper is to prove that the unit signature rank of $\QQ(\zeta_m)^+$ 
tends to infinity with $m$.  We do this by demonstrating an explicit lower bound for the signature 
rank of the subgroup of {\it cyclotomic units} of $\QQ(\zeta_m)^+$.  
We then show that 
the difference between the signature rank of the units 
in the maximal real subfield of the cyclotomic field of $m p_1^{n_1} \dots p_s^{n_s}$-th roots of unity 
(with all $p_i$ odd) and its maximum possible
value is bounded independent of $n_1, \dots , n_s$ (and is constant if all the $n_i$ are sufficiently large).
This in particular provides infinitely many families of cyclotomic fields whose unit signature ranks are `nearly maximal'.
Finally, we show this difference can be arbitrarily large,
conditional on the existence of infinitely many cyclic cubic fields with totally positive fundamental units.

\section{Signatures}

For any totally real field $F$ of degree $n$ over $\QQ$ 
let $F_\RR=F \otimes_{\QQ} \RR \simeq \prod_{\text{$v$ real}} \RR$
where the product is taken over the real embeddings $v$ of $F$.
Define the {\it archimedean signature space} $V_{\infty,F}$ of $F$ to be 
\begin{equation*} 
V_{\infty,F} = F_\RR^* / F_\RR^{*2} \simeq \prod_{\text{$v$ real}} \{\pm 1\} \simeq  \FF_2^{n},
\end{equation*}
where by identifying $\{\pm 1\}$ with the finite field $\FF_2$ of two elements we 
view the multiplicative group $V_{\infty,F}$ as a vector space over $\FF_2$, written additively.

For any $\alpha \in F^*$ and $v:F \hookrightarrow \RR$ a real place of $F$, let $\alpha_v=v(\alpha)$ and 
define the sign of $\alpha_v$ as usual by $\sign(\alpha_v)=\alpha_v/\lvert \alpha_v \rvert \in \{\pm 1\}$.  
Write $\sgn(\alpha_v) \in \FF_2$ for $\sign(\alpha_v)$ when viewed in the additive group $\FF_2$, i.e.,
$\sgn(\alpha_v) = 0$ if $\alpha_v > 0$ and $\sgn(\alpha_v) = 1$ if $\alpha_v < 0$.
The {\it (archimedean) signature map} of $F$ is the homomorphism
\begin{equation*}
\begin{aligned}
\sgn_{\infty,F}: F^* &\to V_{\infty,F} \\
\alpha &\mapsto (\sgn(\alpha_v))_{v}.
\end{aligned}
\end{equation*}

In the case when $F/\QQ$ is Galois and one real embedding of $F$ is fixed, 
we can index the real embeddings of $F$ by the elements $\sigma$ in
$\Gal(F/\QQ)$, and 
\begin{equation*} 
\sgn_{\infty,F} (\alpha) = ( \sgn( \sigma( \alpha) ))_{\sigma \in \Gal(F/\QQ)} ,
\end{equation*}
where $\sgn$ is the sign (viewed as an element of $\FF_2$) in the fixed real embedding.
The element $\sgn_{\infty,F} (\alpha)$ is called the {\it signature} of $\alpha$.

The collection of all the signatures $\sgn_{\infty,F} (\epsilon)$ where $\epsilon$ varies
over the units of $F$  is called the {\it unit signature group} of $F$; the rank of
this subspace of $V_{\infty,F}$ is called the {\it (unit) signature rank} of $F$ and,
as previously mentioned, is a measure of how many different possible sign configurations arise from
the units of $F$.

Define the {\it (unit signature rank) ``deficiency'' of $F$}, denoted $\delta(F)$, to be 
the corank of the unit signature group of $F$ in $V_{\infty,F}$, i.e., 
$[F:\QQ]$ minus the signature rank of the units of $F$.  The deficiency of $F$ is just 
the nonnegative difference between the
unit signature rank of $F$ and its maximum possible value---the deficiency is 0 if and only if there
are units of every possible signature type.  The deficiency is also the rank of the
group of totally positive units of $F$ modulo squares.

\begin{remark}\label{rem:deficiencies}
For any finite extension $L/F$ of totally real fields we have $\delta (F) \le \delta(L)$, a result due to
Edgar, Mollin and Peterson (\cite[Theorem 2.1]{EMP}).  We briefly recall the reason: the intersection of 
the Hilbert class field $H_L$ of $L$ with the
strict (or narrow) Hilbert class field $H_F^{\textup{st}}$ of $F$ is easily seen to be the Hilbert class field 
$H_F$ of $F$ since $L$ is totally real.  
The composite field $H_L H_F^{\textup{st}}$ is a subfield of the strict
Hilbert class field, $H_L^{\textup{st}}$, of $L$ 
and has degree over $H_L$ equal to $[H_F^{\textup{st}} : H_F]$ because  $ H_F =   H_L \cap H_F^{\textup{st}} $.
Since $[H_F^{\textup{st}} : H_F] = 2^{\delta (F)}$ and $[H_L^{\textup{st}} : H_L] = 2^{\delta (L)}$
(see \cite[\S 2]{D-V} for details), the result follows.
$$
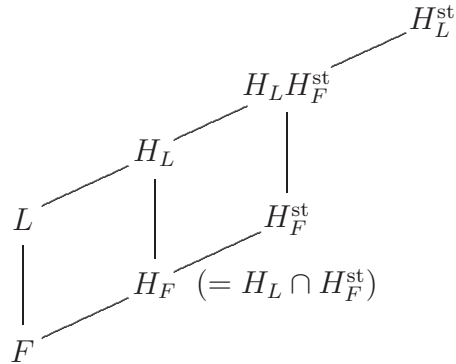
\beginpicture
\setcoordinatesystem units <1 pt, 1 pt>
\linethickness= 0.3pt
%
\put {$F$} at  0 0
\put {$L$} at  0 50 
\put {$H_F$} at  50 25      
\put {$H_L$} at  50 75    
\put { ($ =   H_L \cap H_F^{\textup{st}} $) } at 100 25
\put {$H_F^{\textup{st}}$} at  100 50      
\put {$H_L H_F^{\textup{st}}$} at  100 100  
\put {$H_L^{\textup{st}}$} at 155 125             
\putrule from 0 10 to 0 40      
\putrule from 50 35 to 50 65      
\putrule from 100 60 to 100 90      
\setlinear \plot  7.5 5.0  40.0 20.0 /      
\setlinear \plot  7.5 55.0  40.0 70.0 /      
\setlinear \plot  57.5 30.0  90.0 45.0 /      
\setlinear \plot  57.5 80.0  85.0 92.7 /      
\setlinear \plot  117 107.5  145.0 120.4 /      
\endpicture
$$
\end{remark}

\section{Circular Units and Signatures in Cyclotomic Fields}

Suppose now that $F = \QQ(\zeta_m )^+$ is the maximal (totally) 
real subfield of the cyclotomic field $\QQ(\zeta_m)$ of
$m$-th roots of unity (with $m$ odd or divisible by 4), which is of degree $\phi(m)/2$ over $\QQ$. 

We can fix an embedding of $\QQ(\zeta_m)$ into $\CC$ by mapping $\zeta_m$ to  $e^{2 \pi i /m}$, which
also fixes an embedding of $\QQ(\zeta_m)^+$ into $\RR$.

The Galois group $\Gal (\QQ(\zeta_m)/\QQ)$ consists of the automorphisms $\sigma_a$ that map $\zeta_m$ to $\zeta_m^a$
for integers $a$, $1 \le a < m$ relatively prime to $m$.  We identify 
$\Gal (\QQ(\zeta_m)^+/\QQ)$ with $\Gal (\QQ(\zeta_m)/\QQ) / \{ 1, \sigma_{-1} \}$ and take as representatives the
elements $ \overline{\sigma_a} \in \Gal (\QQ(\zeta_m)^+/\QQ)$ with $1 \le a < m/2$ relatively prime to $m$.

To get information on the unit signature rank of $\QQ(\zeta_m)^+$ we consider the signatures of the
subgroup of {\it circular} (or {\it cyclotomic}) {\it units}, denoted $C_{\QQ(\zeta_m )}$, 
which when $m$ is a prime power has a set of generators given by $-1$ and 
the $\phi(m)/2 -1 $ independent elements
\begin{equation} \label{eq:circunits}
\zeta_m^{(1-a)/2} \dfrac{1 - \zeta_m^a}{1 - \zeta_m}
\end{equation}
for $1 < a < m/2$ and $a$ relatively prime to $m$ \cite[Lemma 8.1]{W} 
(for $m$ not a prime power the definition of $C_{\QQ(\zeta_m )}$ is more complicated, see \cite[Chapter 8]{W}).  

The group $C_{\QQ(\zeta_m )}$ is contained in $\QQ(\zeta_m)^+$ and when $m$ is a prime power is  
isomorphic (as an additive abelian group) to $\ZZ / 2 \ZZ \times \ZZ^{\phi(m)/2 - 1}$, 
with $-1$ and the elements in \eqref{eq:circunits} serving as independent generators over $\ZZ$.  

If we choose a particular ordering of the elements of $\Gal(\QQ(\zeta_m)^+ / \QQ)$ and an ordering of a set of
generators for the cyclotomic units, then their corresponding signatures in $\FF_2^{\phi(m)/2}$ define
the rows of a $\phi(m)/2 \times \phi(m)/2$ matrix over $\FF_2$, whose rank is the rank of the subgroup
$ \sgn_{\infty,\QQ(\zeta_m)^+} (C_{\QQ(\zeta_m )})$ of signatures of the cyclotomic units.  

We first consider the case where $m$ is an odd prime power.

\subsection*{The case ${\bf m = p^n}$ for an odd prime ${\bf p}$}

When $m = p^n$ for an odd prime $p$ every primitive $m$-th root of unity is the square of 
another primitive $m$-th root of unity, so
the circular units in \eqref{eq:circunits} can be written in the form
\begin{equation} \label{eq:circunitsoddp}
 \dfrac{\zeta_m^a - \zeta_m^{-a}}{\zeta_m - \zeta_m^{-1}}
\end{equation}
for $1 < a < m/2$ and $a$ relatively prime to $m$.  If we order the elements in
$ \Gal (\QQ(\zeta_m)^+/\QQ) $ by 
$ \overline{\sigma_b}$ with $1 \le b < m/2$ relatively prime to $m$,
then the signature of the element in \eqref{eq:circunitsoddp} is given by the signs
of the elements
\begin{equation*} 
\overline{\sigma_b} \left ( \dfrac{\zeta_m^a - \zeta_m^{-a}}{\zeta_m - \zeta_m^{-1}}  \right ) 
=  
\dfrac{\zeta_m^{ab} - \zeta_m^{-ab}}{\zeta_m^b - \zeta_m^{-b}} , \quad 1 \le b < m/2, \text{$(b,m) = 1$} .
\end{equation*}
Under the embedding $\zeta_m = e^{2 \pi i /m}$ we have

\begin{equation} \label{eq:circunitsab}
\dfrac{\zeta_m^{ab} - \zeta_m^{-ab}}{\zeta_m^b - \zeta_m^{-b}} = \dfrac{  \sin (2 \pi a b /m)    }   { \sin (2 \pi b /m)   }.
\end{equation}
Since $1 \le b < m/2$, the denominator $\sin (2 \pi b /m) $ is positive.  Hence
$\overline{\sigma_b} ( ( \zeta_m^a - \zeta_m^{-a} )/(  \zeta_m - \zeta_m^{-1} ))$ is positive if and only if 
$ \sin (2 \pi a b /m) $ is positive, which happens precisely when the least positive residue of $a b$ modulo $m$
is in $(0, m/2)$.  

It follows that the rank of the subspace $ \sgn_{\infty,{\QQ(\zeta_m )}^+} (C_{\QQ(\zeta_m )})$ of $V_{\infty,\QQ(\zeta_m)^+}$ is 
equal to the rank of a $\phi(m)/2 \times \phi(m)/2$ matrix $C = (c_{a,b})$ over $\FF_2$
(referred to as the {\it circular unit signature matrix})
whose rows are indexed by the elements $a$ relatively prime to $m$ with $1 \le a < m/2 $ and whose columns are 
indexed by the elements $b$ relatively prime to $m$ with $1 \le b < m/2 $, as follows.
The first row of $C$ is $(1,1, \dots, 1)$, corresponding to the signs $(-1,-1, \dots, -1)$ of the element $-1$, 
viewed in the additive group $\FF_2$, so 
\begin{equation*}
c_{1,b}  = 1 ,
\end{equation*}
for $1 \le b < m/2 $, $b$ relatively prime to $m$.  For
$2 \le a < m/2 $, $1 \le b < m/2$, $a$ and $b$ relatively prime to $m$, we have
\begin{equation*}
c_{a,b}  = 
\begin{cases}
0 & \text{if $ a b \ (\text{mod } m)    \in (0, m/2)$ } \\ 
1 & \text{if $ a b \ (\text{mod } m)    \in (m/2, m)$ } \\ 
\end{cases} ,
\end{equation*}
where $a b \ (\text{mod } m)$ is taken to be the least positive residue of $a b$ modulo $m$. 

For computational purposes, it is useful to note the row indexed by $a$ is the $i$-th row of the
matrix where $i = a - \lfloor (a-1)/p \rfloor$ and the $i$-th row is indexed by $a = i  + \lfloor (i-1)/ (p-1) \rfloor$
(and similarly for the numbering of the columns).

If we add the first row of $C$ to the remaining rows (which does not affect the rank of the
matrix), we obtain the {\it modified circular unit signature matrix} $M = (c_{a,b}')$ with
\begin{equation*}
c_{a,b}' = 
\begin{cases}
1 & \text{if $ a b \ (\text{mod } m)    \in (0, m/2)$ , and } \\ 
0 & \text{if $ a b \ (\text{mod } m)    \in (m/2, m)$ ,} \\ 
\end{cases} 
\end{equation*}
for $1 \le a < m/2 $ and $1 \le b < m/2$, $a$ and $b$ relatively prime to $m$ and as before
$a b \ (\text{mod } m)$ is taken to be
the least positive residue of $a b$ modulo $m$ (this matrix appears in \cite{Da} in the case when $m$ is an odd prime).

The entry $c_{2^d,b}'$ of the 
matrix $M$ in the column indexed by $b$ ($b = 1, \dots, m/2$, $b$ prime to $m$) and the row indexed by $2^d$ ($1 \le 2^d < m/2$ )  is 1 if the least
positive residue of $2^d b$ modulo $m$ lies in $(0, m/2)$ and is 0 if the least positive residue lies in $(m/2,m)$.  
Writing $ 2^d b = A m + r $ with an integer $A$ 
and least positive remainder $r$ with $0 \le r < m$,
i.e., $ 2^{d+1} b / m = 2A + (2 r /m)$, it follows that $ r \in (0,m/2)$ implies $\lfloor 2^{d+1} b / m \rfloor = 2A$ and 
$ r \in (m/2,m)$ implies $\lfloor 2^{d+1} b / m \rfloor = 2A +1$.  
As a consequence, the entry of $M$ in the row indexed by $2^d$ and column indexed by $b$ 
is 0 if $ \lfloor 2^{d+1} b / m \rfloor $ is odd and is 1 if $ \lfloor 2^{d+1} b / m \rfloor $ is even.

\begin{nonumlemma} \label{lem:indeppcols}
Suppose $m = p^n$ where $p$ is an odd prime and $n \ge 1$.  Let $k \ge 1$ be any integer with $ m > 2^{k+2} $.  
If $b_0(k) = \lfloor \dfrac{(2^k - 2)m}{2^{k+1}} \rfloor + 1$ and
$b_1(k) = \lfloor \dfrac{(2^k - 2)m}{2^{k+1}} \rfloor + 2$, then
$ \lfloor 2^{d+1} b_0(k) / m \rfloor $ and $ \lfloor 2^{d+1} b_1(k) / m \rfloor $  are both odd for 
$ d = 1,2, \dots, k-1$ and both even for $d = k$.
\end{nonumlemma}

\begin{proof}
Write 
\begin{equation*} 
\dfrac{(2^k - 2)m}{2^{k+1}} = \lfloor \dfrac{(2^k - 2)m}{2^{k+1}} \rfloor + \theta
\end{equation*}
where $0 \le \theta < 1$.  Then
\begin{equation*} 
b_0(k) = \dfrac{2^k - 2}{2^{k+1}} m + (1 -\theta) \quad \text{and} \quad b_1(k) = \dfrac{2^k - 2}{2^{k+1}} m + (2 -\theta),
\end{equation*}
so
\begin{equation*}
\dfrac{ 2^{d+1} b_0(k)}{m}  = 2^d - \dfrac{2}{2^{k-d}}  + \dfrac{2^{d+1}}{m} (1 -\theta) 
\quad \text{and} \quad
\dfrac{ 2^{d+1} b_1(k)}{m}  = 2^d - \dfrac{2}{2^{k-d}}  + \dfrac{2^{d+1}}{m} (2 -\theta).
\end{equation*}
Since $0 < 1 - \theta \le 1$, $1 < 2 - \theta \le 2$, and $m > 2^{k+2}$, when $d = k$ this gives
\begin{equation*} 
\lfloor 2^{k+1} b_0(k) / m \rfloor = \lfloor 2^{k+1} b_1(k) / m \rfloor = 2^k - 2,
\end{equation*}
so $ \lfloor 2^{k+1} b_0(k) / m \rfloor $ and $ \lfloor 2^{k+1} b_1(k) / m \rfloor $  are both even.
For $ 1 \le d < k$, we have 
\begin{equation*} 
2^d - 1 < 2^d - \dfrac{2}{2^{k-d}}  + \dfrac{2^{d+1}}{m} (1 -\theta) < 2^d - \dfrac{2}{2^{k-d}}  + \dfrac{2^{d+1}}{m} (2 -\theta) < 2^d , 
\end{equation*}
so that
\begin{equation*} 
\lfloor 2^{d+1} b_0(k) / m \rfloor = \lfloor 2^{d+1} b_1(k) / m \rfloor = 2^d - 1,
\end{equation*}
and $ \lfloor 2^{d+1} b_0(k) / m \rfloor $ and $ \lfloor 2^{d+1} b_1(k) / m \rfloor $  are both odd, completing the proof of the lemma.
\end{proof}

We can use the lemma to give the following lower bound for the number of independent signatures
for the circular units in this case.

\begin{proposition} \label{prop:pcase}
Suppose $p$ is an odd prime and $n \ge 1$.  Then the rank of the group of signatures of the circular units
in $\QQ(\zeta_{p^n})^+$ is at least $\lfloor \log_2 (p^n) \rfloor - 2$.
\end{proposition}

\begin{proof}
Since $b_0(k)$ and $b_1(k)$ in the lemma differ by 1 and $m = p^n$, at least one is relatively prime to $m$.  It then  
follows from the lemma that for each $k \ge 1$ with $2^{k+2} < m$ there is a $B(k)$,
relatively prime to $m$ and satisfying $1 \le B(k) < m/2$, 
such that $ \lfloor 2^{d+1} B(k) / m \rfloor $ is odd for $ d = 1,2, \dots, k-1$ and even for $d = k$.

By the remarks before the lemma, it follows that for each $k \ge 1$ with $2^{k+2} < m$,  
the entries of the matrix $M$ in the column indexed by $B(k)$ and belonging to the rows indexed by $2,4,8,\dots, 2^k$ 
are $0,0,\dots,0,1$, respectively. 
In particular, this shows that the row of $M$ indexed by $2^k$ is not in the span of the rows indexed by $2,4,\dots,2^{k-1}$.  
Applying this successively for $k = 1,2, \dots, \lfloor \log_2 m \rfloor - 2 $ shows that all these
rows are linearly independent, implying that the rank of $M$ is at least $\lfloor \log_2 m \rfloor - 2$, which proves the
proposition.
\end{proof}

\subsection*{The case ${\bf m = 2^n}$, ${\bf n \ge 2}$}

In this case let $\zeta_{2^{n+1}}$ be a $2^{n+1}$-st root of unity with $\zeta_{2^{n}} = {\zeta_{2^{n+1}}^2}$.  Fix an
embedding into $\CC$ mapping $\zeta_{2^{n+1}}$ to $e^{2 \pi i /2^{n+1}}$ and order the elements of
$ \Gal ({\QQ(\zeta_m )}^+/\QQ) $
as $ \overline{\sigma_b}$ with $b = 3,5, \dots, 2^{n-1} - 1$, $b$ odd.  Then
the conjugates of the elements in \eqref{eq:circunits} are given by
\begin{equation} \label{eq:2powercircunits}
\dfrac{ \zeta_{2^{n+1}}^{ab} - \zeta_{2^{n+1}}^{-ab} } { \zeta_{2^{n+1}}^{b} - \zeta_{2^{n+1}}^{-b} } = \dfrac{\sin ( \pi a b /2^{n} ) } {\sin ( \pi b /2^{n} ) }
\end{equation}
for $1 < a < 2^{n-1}$,  $1 \le b < 2^{n-1}$, with $a$ and $b$ both odd.  The sign of the unit in 
\eqref{eq:2powercircunits} is $+1$ if the least positive residue of $a b$ modulo $2^{n+1}$ lies in $(0,2^n)$ and
is $-1$ if the least positive residue lies in $(2^n, 2^{n+1})$.

In this case the circular unit signature matrix has first row consisting of all 1's 
and the entry in the row indexed by $a$ and column indexed by $b$ is given by 
\begin{equation} \label{eq:sigmatrix2one}
c_{1,b}  = 1 
\end{equation}
for $b$ odd, $1 < b < 2^{n-1}$,  and 
\begin{equation} \label{eq:sigmatrix2two}
c_{a, b}  = 
\begin{cases}
0 & \text{if $ a b \ (\text{mod } 2^{n+1})    \in (0,2^n)$ , and } \\ 
1 & \text{if $ a b \ (\text{mod } 2^{n+1})    \in (2^n, 2^{n+1})$ } \\ 
\end{cases} ,
\end{equation}
for $a$ and $b$ odd, $1 < a < 2^{n-1} $ and $1 \le b < 2^{n-2} $, where
$a b \ (\text{mod } 2^{n+1})$ is taken to be the least positive residue of $a b$ modulo $2^{n+1}$.

An argument similar to that for odd prime powers (here for the
rows indexed by $2^d - 1$ and column indexed by $2^n - 2^{n-k+1} + 1$) shows the rank
of the circular unit signature matrix is at least $n - 2 = \lfloor \log_2 m \rfloor - 2$, but for
this case Weber in 1899 proved the following definitive result.

\begin{proposition} \label{prop:2case}
(Weber, \cite[B, p.~821]{We}) Suppose $n \ge 2$.  Then the rank of the group of signatures of the circular units
in the maximal totally real subfield of the cyclotomic field of $2^n$-th roots of unity is maximal, i.e.,
equal to $2^{n-2}$.

\end{proposition}

This result was generalized by Hasse in \cite{Ha}, whose simpler and more conceptual proof used the fact
that the circular unit $( \zeta_{2^{n+1}}^{5} - \zeta_{2^{n+1}}^{-5} ) / ( \zeta_{2^{n+1}} - \zeta_{2^{n+1}}^{-1} ) $ in
${\QQ(\zeta_{2^n} )}^+$ has norm $-1$ and showed the signatures of its Galois conjugates generate 
a group of maximal signature rank (see \cite[pp.~28-29]{Ha}).
A nice proof of this, using the fact that over $\FF_2$ the only irreducible representation of a 2-group is
the trivial representation, can be found in \cite{Ga}.  Another nice proof of Weber's result can be found in \cite{D}.

\begin{remark}
Unlike the situation for the 2-power cyclotomic fields, not every possible signature type occurs for the circular units
in general cyclotomic fields, even for $m = p$ an odd
prime (for example, in the case $p = 29 $ the circular unit signature group has rank 11, not the maximal possible
rank of 14 \cite[Appendix I, p.~ 70]{Da}). 
\end{remark}

\section{Signatures in composites of extensions}

Propositions  \ref{prop:pcase} and \ref{prop:2case} already imply that the signature rank of the
units of ${\QQ(\zeta_m )}^+$ tends to infinity with $m$ (since ${\QQ(\zeta_{p^n} )}^+ \subset {\QQ(\zeta_{m} )}^+$ if $p^n$ divides $m$ and
the largest prime power divisor of $m$ tends to infinity as $m$ tends to infinity), but this can be
made more precise using the following result, which may be of independent interest.

Suppose $F/\QQ$ and $F'/\QQ$ are two totally real Galois extensions of $\QQ$ with $F \cap F' = \QQ$.
Then the composite field $F F'$ has Galois group  $\Gal(FF'/\QQ) = \Gal(F/\QQ) \times \Gal(F'/\QQ)$,
and independent signatures in $F$ and $F'$ produce essentially independent signatures in $F F'$:

\begin{proposition} \label{prop:compositeranks}
With $F$ and $F'$ as above, suppose $\alpha_1, \dots, \alpha_r$ are nonzero elements of $F$ whose signatures 
$\sgn_{\infty,F}(\alpha_i)$, $i = 1,\dots,r$ are linearly independent in the
$\FF_2$-vector space $V_{\infty,F}$.  Suppose similarly that $\beta_1, \dots, \beta_s$ are 
nonzero elements of $F'$ whose signatures $\sgn_{\infty,F'}(\beta_j)$, $j = 1, \dots, s$,  are linearly independent in the
$\FF_2$-vector space $V_{\infty,F'}$.  
Then the dimension of the space generated by the signatures of $\alpha_1, \dots, \alpha_r,\beta_1, \dots, \beta_s$ 
in the $\FF_2$-vector space $V_{\infty,FF'}$ is $r + s$ unless the subgroups generated by $\alpha_1, \dots, \alpha_r$ and
by $\beta_1, \dots, \beta_s$ both contain totally negative elements, in which case the dimension is $r + s - 1$.  
\end{proposition}

\begin{proof}
If the signatures in $V_{\infty,FF'}$ of $\alpha_1, \dots, \alpha_r, \beta_1, \dots, \beta_s$ are linearly 
dependent, then there is an element ${\alpha_1}^{a_1} \dots {\alpha_r}^{a_r} {\beta_1}^{b_1} \dots {\beta_s}^{b_s}$ in $FF'$
with $a_1,\dots,a_r,b_1, \dots, b_s \in \{0,1\}$, not all 0, that is totally positive.  Assume without loss
that at least one of $a_1,\dots,a_r$ is not 0, and let
$\alpha = {\alpha_1}^{a_1} \dots {\alpha_r}^{a_r}$ and $\beta = {\beta_1}^{b_1} \dots {\beta_s}^{b_s}$.  Then
$\sigma \tau (\alpha \beta) = \sigma (\alpha) \tau (\beta)$ is positive for every $\sigma \in  \Gal(F/\QQ)$ and every
$\tau \in \Gal(F'/\QQ)$.
Since the signatures $\sgn_{\infty,F}(\alpha_i)$, $i = 1,\dots,r$ are linearly independent and some $a_i$ is nonzero,  there exists a 
$\sigma_0$ such that $\sigma_0 (\alpha)$ is negative.  This implies $\tau (\beta)$ is negative for every $\tau$, i.e.,
$\beta$ is totally negative. Then, since there is a $\tau_0$ with $\tau_0 (\beta)$ negative, it follows that also
$\sigma (\alpha)$ is negative for every $\sigma$, i.e., $\alpha$ is totally negative.  Hence there is at most
one nontrivial relation among the signatures of $\alpha_1, \dots, \alpha_r, \beta_1, \dots, \beta_s$ in $V_{\infty,FF'}$---this
nontrivial relation occurs if and only if $(1,1,\dots,1) \in V_{\infty,F}$ is in the $\FF_2$-space spanned 
by $\sgn_{\infty,F}(\alpha_i)$, $i = 1,\dots,r$
and $(1,1,\dots,1) \in V_{\infty,F'}$ is in the $\FF_2$-space spanned by $\sgn_{\infty,F'}(\beta_j)$, $j = 1,\dots,s$, completing
the proof.
\end{proof}

\section{The signature rank of the units in $\QQ(\zeta_m)$ for general $m$}

If we combine Propositions \ref{prop:pcase} and \ref{prop:2case} on the signature ranks in the prime power
case with Proposition \ref{prop:compositeranks} we obtain the following result.

\begin{theorem} \label{thm:1}
Suppose the positive integer $m$ is odd or is divisible by 4.   
Then the rank of the group of signatures of the group of circular units 
in $\QQ(\zeta_m)^+$ is at least
$\log_2 (m) - 4 \omega (m) +1 $, where $\omega(m)$ is the number of distinct prime factors of $m$. 
In particular, the signature rank of the units in $\QQ(\zeta_m)^+$ 
tends to infinity with $m$.  
\end{theorem}

\begin{proof}
Write $m = p_1^{a_1} \dots p_k^{a_k}$. Then $\QQ(\zeta_m)^+$ contains
the composite of the totally real fields $\QQ( \zeta_{{p_i}^{a_i}} )^+$, $i = 1, \dots, k$.
By Propositions  \ref{prop:pcase} and \ref{prop:2case}, the signature rank of the circular units
in these latter fields is at least $\lfloor \log_2 ({p_i}^{a_i}) \rfloor - 2$ (and much better
when $p = 2$).  Applying the previous proposition
repeatedly shows that the signature rank of the group generated by the circular units 
is at least
\begin{equation*}
\sum_{i=1}^k \big ( \lfloor \log_2 ({p_i}^{a_i}) \rfloor - 2 \big ) - (k-1).
\end{equation*}
Since $\lfloor \log_2 ({p_i}^{a_i}) \rfloor > \log_2 ({p_i}^{a_i}) - 1$, the lower bound in the
theorem follows.   The final statement in the theorem follows from standard bounds on the growth
of $\omega(m)$ (see, for example, \cite[Section 22.10]{H-W}).
\end{proof}

\begin{corollary}
With the exception of $m = 12$, no maximal real subfield of any cyclotomic field of $m$-th roots of unity 
has a fundamental system of units that are all totally positive.
\end{corollary}

\begin{proof}
The signature rank of the circular units is at least 2 for $m = 2^3, 3^2, 5, 7, 11, 13$ by direct computation and
for all $m = p$ for primes $p \ge 17$ by Proposition \ref{prop:pcase}.  It follows that the signature
rank of the circular units is at least 2 for all $m$ divisible by $2^3, 3^2$ or any odd prime $p \ge 5$.  The only
remaining possible values for $m$ are $m = 3, 4, 12$.  The first two have no units of infinite order, and the third
has maximal real subfield $\QQ(\sqrt{3})$ with totally positive fundamental unit $2 + \sqrt 3$.  
\end{proof}

\section{Signatures in Cyclotomic Towers over Cyclotomic Fields} \label{sec:towers}

Computations suggest that the signature rank of the units in the real subfield of
the cyclotomic field of $m$-th roots of unity is in fact always
close to the maximal possible rank of $\phi(m)/2$ (equivalently, the unit signature rank
deficiency for these fields should be close to 0), i.e., nearly all possible signature types arise for units.
This is in keeping with the heuristics
in \cite{D-V} suggesting that `most' totally real fields have nearly maximal unit signature rank
(although these abelian extensions are hardly `typical').  

In this section we prove that for infinitely many different families of cyclotomic fields the units 
do indeed have nearly maximal
signature rank.  We do this by showing the 
unit signature rank deficiency is bounded in (finitely many composites of)
prime power cyclotomic towers over cyclotomic fields.

\begin{theorem} \label{thm:2}
Suppose $p_1, \dots , p_s$ ($s \ge 1$) are distinct odd primes and suppose $m$ 
is any positive integer that is either odd or divisible by 4 and that is relatively prime to $p_1, \dots , p_s$.  

Let $\delta(m; n_1, \dots , n_s) = \delta( \QQ(\zeta_{m p_1^{n_1} \dots p_s^{n_s}})^+ ) \ge 0$ denote
the unit signature rank deficiency of the maximal real subfield of the cyclotomic field of 
$m p_1^{n_1} \dots p_s^{n_s}$-th roots of unity defined in \S 2, i.e., the
nonnegative difference between the signature rank of the units
of $\QQ(\zeta_{m p_1^{n_1} \dots p_s^{n_s}})^+ $ and its maximum possible value $\phi( m p_1^{n_1} \dots p_s^{n_s})/2$.
Then

\begin{enumerate}

\item[(a)]
$\delta(m; n_1, \dots , n_s) \le \delta(m; n'_1, \dots , n'_s)$ if $n_i \le n'_i$ for all $1 \le i \le s$,

\item[(b)]
$\delta(m; n_1, \dots , n_s)$ is bounded independent of $n_1, \dots , n_s$, and

\item[(c)]
$\delta(m; n_1, \dots , n_s)$ is constant (depending on $m$) if $n_1, \dots , n_s$ are all sufficiently large.

\end{enumerate}

\end{theorem}

\begin{proof}
The maximal real subfield of the cyclotomic field of $m p_1^{n_1} \dots p_s^{n_s}$-th roots of unity is a
subfield of the maximal real subfield of the cyclotomic field of $m p_1^{n'_1} \dots p_s^{n'_s}$-th roots of unity
if $n_i \le n'_i$, $i = 1, \dots ,s$, so (a) follows immediately from Remark \ref{rem:deficiencies}.

Suppose that $n_i \ge 1$ for $1 \le i \le s$ and let $L$ denote the cyclotomic field $\QQ(\zeta_{m p_1^{n_1} \dots p_s^{n_s}})$, 
with maximal real subfield $L^+ = \QQ(\zeta_{m p_1^{n_1} \dots p_s^{n_s}})^+ $. 
Then the strict class number of $L^+ $ divides twice the class number of $L$, as follows.
The quadratic extension $L/L^+$ is ramified at a finite prime (if $m = s = 1$) or is
ramified only at infinity (otherwise).  Hence, if $H_{L^+}^{\textup{st}}$ is the strict Hilbert class field of $L^+$,
then the degree over $L$ of the composite $L H_{L^+}^{\textup{st}}$ is either 
the strict class class number of $L^+$ (if $m = s = 1$) or half that (otherwise).
Since $L H_{L^+}^{\textup{st}}$ is an abelian extension of $L$ that is unramified at finite primes, it is contained 
in the Hilbert class field of the complex field $L$, so $ [ L H_{L^+}^{\textup{st}} : L]$ divides the
class number of $L$, which gives the desired divisibility.

Next observe that the cyclotomic fields $\QQ(\zeta_{m p_1^{n_1} \dots p_s^{n_s}})$ with $n_i \ge 1$ for $1 \le i \le s$
are the subfields of the composite of the cyclotomic $\ZZ_{p_i}$-extensions of $\QQ(\zeta_{m p_1 \dots p_s})$ 
for $1 \le i \le s$.  
Hence the 2-primary part of the class number of $\QQ(\zeta_{m p_1^{n_1} \dots p_s^{n_s}})$
is bounded for all $s$-tuples $(n_1, \dots , n_s)$ and is
constant if $n_1, \dots , n_s$ are all sufficiently large 
by a theorem of Friedman (\cite{F}) extending a result of Washington.  
By the previous observation, the strict class number of $L^+ = \QQ(\zeta_{m p_1^{n_1} \dots p_s^{n_s}})^+ $
is therefore bounded for all $s$-tuples $(n_1, \dots , n_s)$ (and is in fact 
constant if $n_1, \dots , n_s$ are all sufficiently large).

Finally, since the strict class number of $L^+$ is the product of the usual class number 
of $L^+$ with $2^{\delta(L^+)}$, we obtain (b).  Then by (a), we obtain (c).
\end{proof}

\begin{corollary} \label{cor:cor1}
With notation as in Theorem \ref{thm:2}, the unit signature rank deficiencies
for all totally real abelian fields $F$ whose conductor is a product of a divisor of $m$ with an integer whose prime divisors are 
among the set $\{ p_1, \dots , p_s \}$, are uniformly bounded.  
\end{corollary}

\begin{proof}
Any totally real abelian field $F$ having conductor $d p_1^{n_1} \dots p_s^{n_s}$ where $d$ is a divisor of $m$ and
with $n_i \ge 0$ for $1 \le i \le s$ is
contained in  $\QQ(\zeta_{m p_1^{n_1} \dots p_s^{n_s}})^+ $.  Hence $\delta(F) \le \delta(m; n_1,\dots,n_s)$ by
Remark \ref{rem:deficiencies}, and the result follows immediately from Theorem \ref{thm:2}.
\end{proof}

\begin{remark}
Corollary \ref{cor:cor1} shows that 
among all the abelian fields whose conductors are supported in a fixed finite set of primes, almost all have
nearly maximal unit signature rank (in the precise sense that the signature rank deficiencies are uniformly bounded by a
constant depending only on the set of primes chosen).  
\end{remark}

We highlight some particular special cases:
 
\begin{corollary} \label{cor:basicZp}
Let $k$ be a finite totally real abelian extension of $\QQ$.  For $p$ an odd prime, let $k^{p,\infty}$ denote the
cyclotomic $\ZZ_p$-extension of $k$.  If $k_n$ is the subfield of $k^{p,\infty}$ of degree $p^n$ over $k$,
then the signature rank of the units of $k_n$ differs from $[k_n:\QQ]$ by a constant amount for $k_n$ sufficiently far up
the tower.
\end{corollary}

Applying Corollary \ref{cor:basicZp} to $k = \QQ(\zeta_p)^+$ for $p$ odd, together with Weber's result in Proposition \ref{prop:2case},
gives the following.

\begin{corollary} \label{cor:primepower}
For any prime $p$, the difference between the signature rank of the units of 
$\QQ(\zeta_{p^n})^+$ and $\phi(p^n)/2$ is constant for $n$ sufficiently large (the constant depending on $p$).
\end{corollary}

Corollary \ref{cor:primepower} gives another proof of the results in \S 3 that the signature ranks of the units in 
the fields $\QQ(\zeta_{p^n})^+$ tend to infinity as $n$ tends to infinity.
Corollary \ref{cor:primepower} is far superior, asymptotically, to the results in \S 3 for the cyclotomic fields of
odd prime power conductor since it shows the signature rank is `nearly' maximal.  
The result yields relatively little information for any specific $\QQ(\zeta_{p^N})^+$, however, 
since the unit signature rank deficiency for this field 
could conceivably be close to $\phi(p^N)/2$
(although, as mentioned, this is not expected to happen).  The only explicit lower bounds for the
unit signature rank for $\QQ(\zeta_{p^n})$ for odd $p$ and, more significantly, for general $\QQ(\zeta_{m})$ (for 
example if $m$ is the product of distinct primes, for which the results in this section have little to say), are 
those in \S 3.

\begin{remark}
We have done some computations of the signature ranks for the subgroup of circular units in towers of prime power cyclotomic fields.
While the computations are somewhat modest (since $\phi(p^n)$ grows rapidly with $n$), these
computations have exhibited the following behavior:  
if the signature rank of the 
circular units in $\QQ(\zeta_p)^+$ is 
$\frac12 \phi(p) - \delta$
($\delta \ge 0$), 
then the signature rank of the circular units in the fields 
$\QQ(\zeta_{p^n})^+$ is $\frac12 \phi(p^n) - \delta$, i.e., the circular unit signature rank deficiency is 
constant and equal to its value in the first layer. 
Whether this behavior persists in general is an extremely interesting question.

We also note that the deficiency of the circular units is at least the deficiency for the full group of units,
but may be strictly larger: for the field $\QQ(\zeta_{163})^+$ the circular unit deficiency 
is 2, while the deficiency for the full group of units is 0 (see \cite{D}).
\end{remark}

\section{Unit Signature Rank Deficiencies in Cyclotomic Fields}\label{sec:unbounded}

In this section we show that the signature rank deficiency in the maximal real subfields of 
cyclotomic fields can be arbitrarily large,
under the assumption that there exist infinitely many cyclic
cubic extensions having a system of totally positive fundamental units.

Suppose $k$ is a cyclic cubic extension of $\QQ$ with totally positive fundamental units
$\epsilon_1, \epsilon_2$.  If $E_k$ is the unit group of $k$, then 
$E_k = \{ \pm 1 \} \times \langle \epsilon_1, \epsilon_2 \rangle$ and the subgroup
$\langle \epsilon_1, \epsilon_2 \rangle$ consists of the totally positive units in $k$.

If the Galois group $G$ of $k$ is generated by $\sigma$, then 
$\langle \epsilon_1, \epsilon_2 \rangle$ is a module for the quotient 
$\ZZ[G]/(\sigma^2 + \sigma + 1)$ of the group ring $\ZZ[G]$ of $G$ since 
$\epsilon_1$ and $\epsilon_2$ both have norm $+1$.  This quotient of the group ring is isomorphic to 
the ring of integers in $\QQ(\sqrt{-3})$, which is a principal ideal domain, and it follows that
$\langle \epsilon_1, \epsilon_2 \rangle \iso \ZZ[G]/(\sigma^2 + \sigma + 1)$ as $G$-modules (and not just
as abelian groups).

Modulo squares, $\langle \epsilon_1, \epsilon_2 \rangle$ is therefore isomorphic to $\FF_2[G]/(\sigma^2 + \sigma + 1)$ as
a module over the group ring $\FF_2 [G]$, hence affords the unique 
irreducible 2-dimensional representation of $\FF_2 [G]$.  In particular,
$G$ acts irreducibly and with no nontrivial fixed points on $\langle \epsilon_1, \epsilon_2 \rangle$
modulo squares.

With these preliminaries, we consider the composite of cyclic cubic fields having a 
system of totally positive fundamental units:

\begin{proposition} \label{prop:cubiccomposites}
Suppose $k_1, \dots  k_n$ are linearly disjoint cyclic cubic fields, each with a totally positive system of
fundamental units, i.e., with unit signature rank deficiency $\delta(k_i)$ equal to 2, $i =1,\dots, n$.  Then the 
unit signature rank deficiency $\delta(k_1 \dots k_n)$ for the composite field $k_1 \dots k_n$ is at least $2 n$.  
\end{proposition}

\begin{proof}
We need to prove there are at least $2n$ totally positive units in $k_1 \dots k_n$ that are 
multiplicatively independent modulo squares (in $k_1 \dots k_n$).  
We proceed by induction on $n$, the case $n = 1$ being trivial.  Suppose by induction that the composite 
$k_1 \dots k_s$ contains $2 s$ totally positive units 
$\epsilon'_1 , \dots , \epsilon'_{2s}$ that are multiplicatively independent modulo squares in $k_1 \dots k_s$. 
By assumption, the field $k_{s+1}$ contains two totally positive units $\epsilon_1, \epsilon_2$ that 
are multiplicatively independent modulo squares in $k_{s+1}$.

Suppose $\epsilon \in \langle \epsilon_1, \epsilon_2 \rangle$ 
and $\epsilon' \in \langle \epsilon'_1 , \dots , \epsilon'_{2s} \rangle$  with 
$\epsilon \epsilon' = \alpha^2$ for some $\alpha$ in the composite field $F = k_1 \dots k_s k_{s+1}$.

Let $\sigma \in \Gal (F/\QQ) $ be a lift of a generator for the cyclic group $\Gal (k_{s+1}/\QQ)$
that is the identity on $k_1 \dots k_s$.  Then
$\sigma(\epsilon) \epsilon' = \sigma(\alpha)^2$, so $\sigma(\epsilon)/\epsilon =  (\sigma(\alpha)/\alpha)^2$
is a square in $F$, hence $k_{s+1} (\sqrt{ \sigma(\epsilon)/\epsilon })$
is a subfield of $F$.  Since $F$
has degree $3^s$ over $k_{s+1}$, $k_{s+1} (\sqrt{ \sigma(\epsilon)/\epsilon })$ cannot be a quadratic extension, so
$\sigma(\epsilon)/\epsilon$ is in fact a square in $k_{s+1}$.  Since $\sigma$ acting on 
$\langle \epsilon_1, \epsilon_2 \rangle$ modulo squares in $k_{s+1}$ has no nontrivial fixed point, it follows that
$\epsilon$ is the square of a unit in $k_{s+1}$.  
Then $\epsilon' = \alpha^2 /\epsilon$ would be a square in $F$, a cubic extension of $k_1 \dots k_s$, and,
as before, this implies that $\epsilon'$ would be a square in $k_1 \dots k_s$.

This shows that the totally positive units $\epsilon_1, \epsilon_2, \epsilon'_1 , \dots , \epsilon'_{2s}$ are multiplicatively independent
modulo squares in $F$, completing the proof by induction. 
\end{proof}

For a cyclic cubic field, either there is a totally positive system of fundamental units or the
units have all possible signatures.  Heuristics, supported by computations, in \cite{B-V-V} suggest
that, when counted by discriminant, there is a positive proportion of cyclic cubic fields of either type.  
Roughly 3\% of cyclic cubic fields (see \cite{B-V-V} for the precise value) are predicted to have
unit signature rank deficiency 2, so in particular there should exist infinitely many such cubic fields that are
linearly disjoint.

\begin{theorem} \label{thm:3}
Suppose, as expected, that there exist infinitely many cyclic cubic fields having a totally positive system of fundamental units.
Then the difference between $\phi(m)/2$ and the unit signature rank of $\QQ(\zeta_m)^+$ can be arbitrarily large. 
\end{theorem}

\begin{proof}
By Proposition \ref{prop:cubiccomposites} and Remark \ref{rem:deficiencies}, to obtain a unit signature rank deficiency 
at least $2n$ it suffices to take the 
cyclotomic field whose conductor is the 
product of the distinct primes (which are congruent to 1 mod 3) dividing 
the conductors of $n$ linearly disjoint cyclic cubic 
fields having totally positive fundamental units.   
\end{proof}

\begin{remark}
The same sort of arguments could be applied to composites of other abelian fields of, for example, odd prime degree.
As in the case for cubic fields, it is expected that there are 
infinitely many such cyclic extensions of $\QQ$ with nonzero deficiencies (see \cite{B-V-V} for specific predictions).
This suggests that the unit signature rank deficiency can increase without bound as one moves `horizontally' among cyclotomic 
fields, that is, over fields $\QQ(\zeta_m)^+$ where $m$ is the product of an increasing number of distinct primes, as opposed to
the results of Section \ref{sec:towers} which show the deficiency is bounded as one moves `vertically' among cyclotomic fields. 
\end{remark}

\section{Remarks on 2-adic Unit Signature Rank Deficiencies} \label{sec:unbounded}

The results above have implications for analogous deficiencies for the `2-adic signatures' of 
units in the sense of \cite[Section 4]{D-V}, which we now very briefly outline. We use the 
notation of \cite{D-V}.  

If $F$ is a totally real field with $[F:\QQ] = n$ then there is a structure theorem for the 
image of the 2-Selmer group, $\Sel_2(F)$, under the 2-Selmer signature map $\phi$ (\cite[Theorem A.13]{D-V}).
The space $\phi(\Sel_2(F))$ is $n$-dimensional over $\FF_2$ and is an orthogonal direct sum
$U \perp S \perp U'$, where $U$ is the subspace of elements whose 2-adic signature is trivial, 
$U'$ the subspace of elements whose archimedean signature is trivial, and $S$ is a diagonal subspace.
Since $F$ is totally real, the dimension of $U'$ is the same as the dimension of $U$ (\cite[Theorem A.13(a)]{D-V}) and 
so $S$ has dimension $n - 2 \dim(U)$.

Suppose now that the unit signature rank deficiency of $F$ is $\delta(F)$.  
Then the set of signatures of units is a subspace of dimension $n - \delta(F)$.  
It follows that the dimension of $U$ is at most $\delta(F)$ (so $S$ has dimension at least $n - 2 \delta(F)$)
and that the dimension of the image $\phi(E_F)$ of the units $E_F$ of $F$ is at least $n - \delta(F)$.  
Then $\dim ( \phi(E_F) + S )$ is at most $\dim \phi(\Sel_2(F)) = n$ and 
\begin{align*}
\dim ( \phi(E_F) \cap S ) & = \dim \phi(E_F)   +  \dim S  -  \dim ( \phi(E_F) + S )  \\
& \ge (n - \delta(F)) + (n - 2 \delta(F)) - n = n - 3 \delta(F).
\end{align*}
Since $S$ is a diagonal subspace, it follows that the dimension of the subspace of 2-adic signatures of the
units of $F$ is at least $n - 3 \delta(F)$.  Hence the  2-adic signature deficiency of the units of $F$ is
at most $3 \delta(F)$.

As a consequence, Theorem \ref{thm:2}(b) and Corollaries \ref{cor:cor1}, \ref{cor:basicZp}, and \ref{cor:primepower}
remain true if the (archimedean) signature rank of the units is replaced by the 2-adic signature rank of the units.



\begin{thebibliography}{99}
%
%
%
%
%
\bibitem{B-V-V}
Breen, B., Varma, I.~ and Voight, J.: {\it Heuristics for
cyclic number fields: totally positive units and narrow class groups}, in preparation.
%
\bibitem{Da}
Davis, D. L.: {\it On the distribution of the signs of the conjugates of the cyclotomic units
in the maximal real subfield of the qth cyclotomic fields, q a prime} (Ph.D.~dissertation), 
California Institute of Technology, Pasadena, California, 1969, available from
http://thesis.library.caltech.edu/9554/ .
%
\bibitem{D}
Dummit, D.: {\it A note on the equivalence of the parity of class numbers and the signature ranks of units in cyclotomic fields},
preprint, 2017.
%
\bibitem{D-V}
Dummit, D.~and Voight, J.: {\it The $2$-Selmer group of a number field and heuristics for narrow class 
groups and signature ranks of units}, to appear in Proc.~London Math.~Soc., 2018.
%
\bibitem{EMP}
Edgar, H., Mollin, R., and Peterson, B.: {\it  Class groups, totally positive units, and squares}, Proc. AMS, {\bf 98} (1986), 33--37.
%
\bibitem{F}
Friedman, E.: {\it  Ideal class groups in basic $\ZZ_{p_1} \times \dots \times \ZZ_{p_s}$-extensions of abelian number fields}, 
Invent. math., {\bf 65} (1981/82), 425--440.
%
\bibitem{Ga}
Garbanati, D.A.: {\it Units with norm $-1$ and signatures of units}, Jour.~f\"ur die Reine und Angewandte Mathematik,
{\bf 283/4} (1976), 164--175.
%
\bibitem{H-W}
Hardy, G.H. and Wright, E.M.: An Introduction to the Theory of Numbers, Sixth Edition, revised by
D.R.~Heath-Brown and J.H.~Silverman, Oxford University Press, 2008.
%
\bibitem{Ha}
Hasse, H.: \"Uber die Klassenzahl abelscher Zahlk\"orper, Berlin, 1952.
%
%
%
%
\bibitem{W}
Washington, L.: Introduction to Cyclotomic Fields, Second Edition, Springer-Verlag, 1997.
%
%
\bibitem{We}
Weber, H.: Lehrbuch der Algebra, vol. II, Third Edition, AMS Chelsea Publishing, 2000 (reprint of
1899 Second Edition).




%
%
\end{thebibliography}
\end{document}